\documentclass[12pt]{article}
\usepackage{a4}
\usepackage{amsthm}
\usepackage{amsfonts}
\usepackage{amssymb}
\usepackage{amsmath}
\usepackage{cite}
\usepackage{epsfig}
\usepackage{mathtools}

\newtheorem{theorem}{Theorem}

\newtheorem{lemma}[theorem]{Lemma}
\newtheorem{proposition}[theorem]{Proposition}
\newtheorem{problem}{Problem}
\newcommand\AAA{{\cal A}}
\newcommand\dd{\,\mbox{d}}
\newcommand\NN{{\mathbb N}}
\newcommand\RR{{\mathbb R}}
\newcommand\cut[1]{\|#1\|_{\square}}
\newcommand\dcut[2]{\delta_{\square}(#1,#2)}
\newcommand\wsigma{\widehat{\sigma}}
\let\oldlceil\lceil
\let\oldrceil\rceil
\renewcommand{\lceil}{\left\oldlceil}
\renewcommand{\rceil}{\right\oldrceil}

\DeclareTextCompositeCommand{\v}{OT1}{l}{l\nobreak\hspace{-.1em}'}
\DeclareTextCompositeCommand{\v}{OT1}{t}{t\nobreak\hspace{-.1em}'\nobreak\hspace{-.15em}}

\begin{document}
\title{Non-bipartite $k$-common graphs\thanks{The work of the first, second and fourth authors has received funding from the European Research Council (ERC) under the European Union's Horizon 2020 research and innovation programme (grant agreement No 648509). This publication reflects only its authors' view; the European Research Council Executive Agency is not responsible for any use that may be made of the information it contains. The first and the fourth were also supported by the MUNI Award in Science and Humanities of the Grant Agency of Masaryk University. The second author was also supported by the Leverhulme Trust Early Career Fellowship ECF-2018-534. The third author was supported by an NSERC Discovery grant. The fifth author was supported by IAS Founders' Circle funding provided by Cynthia and Robert Hillas.}}

\author{Daniel Kr{\'a}\v{l}\thanks{Faculty of Informatics, Masaryk University, Botanick\'a 68A, 602 00 Brno, Czech Republic, and Mathematics Institute, DIMAP and Department of Computer Science, University of Warwick, Coventry CV4 7AL, UK. E-mail: {\tt dkral@fi.muni.cz}.}\and
        Jonathan A. Noel\thanks{Department of Mathematics and Statistics, University of Victoria, David Turpin Building A425, 3800 Finnerty Road,  Victoria, Canada V8P 5C2. Previous affiliation: Mathematics Institute and DIMAP, University of Warwick, Coventry CV4 7AL, UK. E-mail: {\tt noelj@uvic.ca}.}\and
	Sergey Norin\thanks{Department of Mathematics and Statistics, McGill University, Montreal, Canada. E-mail: {\tt sergey.norin@mcgill.ca}.}\and
        Jan Volec\thanks{Department of Mathematics, Faculty of Nuclear Sciences and Physical Engineering, Czech Technical University in Prague, Trojanova 13, 120 00 Prague, Czech Republic. Previous affiliation: Faculty of Informatics, Masaryk University, Botanick\'a 68A, 602 00 Brno, Czech Republic. E-mail: {\tt jan@ucw.cz}.}\and
	Fan Wei\thanks{Department of Mathematics, Princeton University, Princeton, NJ, USA. Previous affiliation: School of Mathematics, Institute for Advanced Study, Princeton, USA. E-mail: {\tt fanw@princeton.edu}.}}

\date{} 
\maketitle

\begin{abstract}
A graph $H$ is $k$-common if the number of monochromatic copies of $H$ in a $k$-edge-coloring of $K_n$ is 
asymptotically minimized by a random coloring. 
For every $k$, we construct a connected non-bipartite $k$-common graph.
This resolves a problem raised by Jagger, \v{S}\v{t}ov\'{\i}\v{c}ek and Thomason [Combinatorica 16 (1996), 123--141].
We also show that a graph $H$ is $k$-common for every $k$ if and only if $H$ is Sidorenko and that
$H$ is locally $k$-common for every $k$ if and only if $H$ is locally Sidorenko.
\end{abstract}

\section{Introduction}
\label{sec:intro}

Ramsey's Theorem states that for every graph $H$ and integer $k\ge 2$,
there exists a natural number $R_k(H)$ such that if $N\geq R_k(H)$, then every $k$-edge-coloring of the complete 
graph $K_N$ with $N$ vertices contains a monochromatic copy of $H$.
We study the natural quantitative extension of this question, which was first considered by Goodman~\cite{Goo59}:
\emph{What is the minimum number of monochromatic copies of $H$ in a $k$-edge-coloring of $K_N$ for large $N$?}

A prevailing theme in Ramsey Theory, dating back to an idea of Erd\H{o}s~\cite{Erd47} from the 1940s,
is that one of the best ways to avoid monochromatic substructures is by coloring randomly.
Therefore, it would be natural to expect the answer to the above question
to be the number of monochromatic copies of $H$ in a uniformly random $k$-edge-coloring of $K_N$.
Following~\cite{JagST96}, we say that a graph $H$ is \emph{$k$-common}
if the uniformly random $k$-edge-coloring of $K_N$ asymptotically minimizes the number of monochromatic copies of $H$.
In other words,
the number of monochromatic (labeled) copies of $H$ in every $k$-edge-coloring of $K_N$ is at least
\[(1-o(1))\frac{N^{|H|}}{k^{\|H\|-1}}\]
where $|H|$ and $\|H\|$ denote the number of vertices and edges of $H$, respectively.
The most well-studied case is that of $2$-common graphs, which are often referred to as \emph{common} graphs;
however, we will always say $2$-common to avoid any ambiguity.

Only a handful of graphs are known to be $2$-common and even fewer are known to be $k$-common for $k\ge 3$.
The well-known Goodman Bound~\cite{Goo59} implies that $K_3$ is $2$-common; another proof was given by Lorden~\cite{Lor62}.
This result led Erd\H{o}s~\cite{Erd62} to conjecture that every complete graph is $2$-common
and Burr and Rosta~\cite{BurR80} to extend the conjecture to all graphs.
We now know that $2$-common graphs are far more scarce than Erd\H{o}s, Burr and Rosta had anticipated,
in particular, every non-bipartite graph is a subgraph of a (connected) graph that is not $2$-common~\cite{Fox07}.
Sidorenko~\cite{Sid89} disproved the Burr--Rosta Conjecture by showing that a triangle with a pendant edge is not $2$-common.
Around the same time, Thomason~\cite{Tho89} showed that $K_p$ is not $2$-common for any $p\geq4$,
thereby disproving the original conjecture of Erd\H{o}s~\cite{Erd62}.
Additional constructions showing that $K_p$ is not $2$-common for $p\geq4$ have since been found~\cite{Tho97,FraR93,Fra02}.
Determining the asymptotics of the minimum number of monochromatic copies of $K_4$ in $2$-edge-colorings of large complete graphs
continues to attract a good amount of attention~\cite{Gir79,Nie,Spe11} and
remains one of the most mysterious problems in extremal graph theory (with no conjectured answer).

Jagger, \v{S}\v{t}ov\'{\i}\v{c}ek and Thomason~\cite[Theorem~12]{JagST96} extended the result from~\cite{Tho89} by showing that
no graph containing a copy of $K_4$ is $2$-common.
On the positive side,
Sidorenko~\cite{Sid89} showed that all odd cycles are $2$-common and
Jagger, \v{S}\v{t}ov\'{\i}\v{c}ek and Thomason~\cite[Theorem~8]{JagST96} that all even wheels are $2$-common.
Additional examples of $2$-common graphs can be obtained by certain gluing operations~\cite{JagST96,Sid96}.
However, these operations do not increase the chromatic number and,
for a long time, no examples of $2$-common graphs with chromatic number greater than three were known.
Only in 2012, the $5$-wheel, which has chromatic number four, was shown to be $2$-common~\cite{HatHKNR12} using
Razborov's flag algebra method~\cite{Raz07}; this result settled a problem of~\cite{JagST96}.

Much less is known about $k$-common graphs for $k\ge 3$.
Cummings and Young~\cite{CumY11} proved that every $3$-common graph is triangle-free, which implies that 
the same is true for $k$-common graphs for any $k\ge 3$ (see Section~\ref{sec:prelim} for details).
The only known examples of $k$-common graphs for $k\ge 3$ are bipartite graphs that were known to be Sidorenko.
Jagger, S\v{t}ov\'i\v{c}ek and Thomason~\cite[Section~5]{JagST96} 
asked about the existence of non-bipartite $k$-common graphs;
no examples of such graphs are known, even for $k=3$.
We resolve this by showing the following.
\begin{theorem}
\label{thm:append}
For every $k\ge 2$, there exists $n_k$ such that, for every $n\ge n_k$,
the graph obtained from $K_{2n,2n}$ by pasting a copy of $C_5$ on every 
second vertex in one of the two parts of $K_{2n,2n}$ is $k$-common.
\end{theorem}
Examples of graphs described in the statement of Theorem~\ref{thm:append} can be found in Figure~\ref{fig:append}.
We remark that one of the key ingredients in the proof of Theorem~\ref{thm:append} is establishing that
such graphs are $k$-common in a certain ``local'' sense (see Lemma~\ref{lm:Kn2-local}),
which is proved using spectral arguments.

\begin{figure}
\begin{center}
\epsfbox{kcommon-1.mps}
\end{center}
\caption{Examples of graphs from the statement of Theorem~\ref{thm:append} for $n=1,2,3$.
         The three graphs are denoted by $K_{2,2,C_5}$, $K_{4,4,C_5}$ and $K_{6,6,C_5}$ in Section~\ref{sec:conn}.}
\label{fig:append}
\end{figure}

As we have already mentioned, there is a close connection between $k$-common graphs and Sidorenko graphs.
We say that a graph $H$ is \emph{Sidorenko}
if the number of copies of $H$ in a graph with edge density $d$
is asymptotically minimized by the random graph with edge density $d$.
Sidorenko's Conjecture~\cite{Sid91,Sid93} famously asserts that every bipartite graph $H$ is Sidorenko;
an equivalent conjecture was made earlier by Erd\H{o}s and Simonovits~\cite{ErdS84}.
It is easy to show that every Sidorenko graph is bipartite and $k$-common for every $k\ge 2$.
There are now many families of bipartite graphs that are known to be Sidorenko,
see, e.g.,~\cite{LiS11,ConFS10,Hat10,Sze15,ConKLL18,ConL17,ConL18};
prior to our work, these graphs were the only known examples of $k$-common graphs for any fixed $k\ge 3$.

The following simple construction of~\cite[Theorem~14]{JagST96} shows that, for
every non-bipartite graph $H$, there exists $k\geq2$ such that $H$ is not $k$-common.
Split the vertices of $K_N$ into $2^{k-1}$ sets of roughly equal size, indexed by $0,\dots,2^{k-1}-1$. 
Color the edges between the $i$-th and $j$-th sets with the color corresponding to the first bit on which $i$ and $j$ 
differ in their binary representations and color the edges inside each set with the color $k$.
Since $H$ is non-bipartite,
the only monochromatic copies of $H$ are inside the sets and
thus their number is $(1+o(1))N^{|H|}2^{-(k-1)(|H|-c)}$,
where $c$ is the number of components of $H$.
Thus, if Sidorenko's Conjecture is true,
then Sidorenko graphs are precisely the graphs that are $k$-common for every $k\ge 2$.
We prove this without the assumption that Sidorenko's Conjecture holds.
\begin{theorem}
\label{thm:Sidorenko}
A graph $H$ is $k$-common for all $k\geq2$ if and only if it is Sidorenko.
\end{theorem}
We also establish the variant of Theorem~\ref{thm:Sidorenko} in the local setting,
i.e., when the edge-coloring is ``close'' to the random edge-coloring.
The notion of locally $k$-common graphs is formally defined in Section~\ref{sec:prelim}. Recall that the 
\emph{girth} of a graph is the length of its shortest cycle.
\begin{theorem}
\label{thm:local}
The following holds for every $k\ge 3$: 
if a graph $H$ has odd girth, then $H$ is not locally $k$-common.
\end{theorem}
Since a theorem of Fox and the last author~\cite{FoxW17} asserts that
all forests and graphs of even girth are locally Sidorenko,
Theorem~\ref{thm:local} implies for every $k\ge 3$ that
a graph $H$ is locally $k$-common if and only if $H$ is locally Sidorenko.
We remark that Theorem~\ref{thm:local} strengthens the result of Cummings and Young~\cite{CumY11} that
no graph containing a triangle is $3$-common by showing that such graphs are not even locally $3$-common.

\section{Preliminaries}
\label{sec:prelim}

In this section, we fix the notation used throughout the paper and present basic properties of $k$-common graphs.
We also introduce some of the terminology of the theory of graph limits.
While all our arguments can be presented for finite graphs,
the language of graph limits allows us not to discuss ``small order'' asymptotic terms.
Our notation and terminology mainly follows that of the monograph of Lov\'{a}sz~\cite{Lov12},
and we refer the reader to~\cite{Lov12} for a more thorough introduction.

We write $\NN$ for the set of all positive integers and
$[k]$ for the set of the first $k$ positive integers, i.e., $[k]=\{1,\ldots,k\}$.
We work with the Borel measures on $\RR^d$ throughout the paper and
if $A\subseteq [0,1]^d$ is a measurable subset of $\RR^d$, we write $|A|$ for its measure.
Graphs that we consider in this paper are finite and simple.
If $G$ is a graph,
then its vertex set is denoted by $V(G)$ and its edge set by $E(G)$;
the cardinalities of $V(G)$ and $E(G)$ are denoted by $|G|$ and $\|G\|$, respectively.
A \emph{homomorphism} from a graph $H$ to a graph $G$
is a function $f:V(H)\to V(G)$ such that $f(u)f(v)\in E(G)$ whenever $uv\in E(H)$.
The \emph{homomorphism density} of $H$ in $G$
is the probability that a random function from $V(H)$ to $V(G)$ is a homomorphism,
i.e., it is the number of homomorphisms from $H$ to $G$ divided by $|G|^{|H|}$.
We denote the homomorphism density of $H$ in $G$ by $t(H,G)$.

A \emph{graphon} is a measurable function $W:[0,1]^2\to [0,1]$ that is symmetric,
i.e., $W(x,y)=W(y,x)$ for all $(x,y)\in [0,1]^2$.
Intuitively, a graphon can be thought of as a continuous variant of the adjacency matrix of a graph.
The graphon that is equal to $p\in [0,1]$ everywhere is called the $p$-constant graphon;
when there will be no confusion, we will just use $p$ to denote such a graphon.
A graphon $W$ is a \emph{step graphon}
if there exist a partition of $[0,1]$ into non-null subsets $A_1,\ldots,A_m$ such that
$W$ is constant on each of the sets $A_i\times A_j$, $i,j\in [m]$.
The sets $A_i$, $i\in [m]$, are called \emph{parts} of the step graphon $W$;
the sets $A_i\times A_j$, $i,j\in [m]$, are \emph{tiles} and
those with $i=j$ are \emph{diagonal tiles}.

The notion of \emph{homomorphism density} extends to graphons by setting
\begin{equation}
t(H,W) := \int_{[0,1]^{V(H)}} \prod_{uv\in E(H)} W(x_u,x_v) \dd x_{V(H)}\label{eq:tHW}
\end{equation}
for a graph $H$ and graphon $W$.
We define the \emph{density} of a graphon $W$ to be $t(K_2,W)$.
The quantity $t(H,W)$ has a natural interpretation in terms of sampling a random graph according to $W$:
for an integer $n$, choose $n$ independent uniform random points $x_1,\ldots,x_n$ from the interval $[0,1]$ and
create a graph with the vertex set $[n]$ by joining the vertices $i$ and $j$ with probability $W(x_i,x_j)$.
The graph constructed in this way is called a \emph{$W$-random graph} and denoted by $G_{n,W}$. 
It can be shown that the following holds for every graph $H$ with probability one:
\[\lim_{n\to\infty}t(H,G_{n,W})=t(H,W).\]
A sequence $(G_i)_{i\in\NN}$ of graphs is \emph{convergent}
if the sequence $(t(H,G_i))_{i\in\NN}$ converges for every graph $H$.
A simple diagonalization argument implies that every sequence of graphs has a convergent subsequence.
We say that a graphon $W$ is a \emph{limit} of a convergent sequence $(G_i)_{i\in\NN}$ of graphs if
\[\lim_{i\to\infty}t(H,G_i)=t(H,W)\]
for every graph $H$.
One of the crucial results in graph limits, due to Lov\'asz and Szegedy~\cite{LovS06}, is that
every convergent sequence of graphs has a limit.
Hence, a graph $H$ is Sidorenko if and only if $t(H,W)\ge t(K_2,W)^{\|H\|}$ for every graphon $W$.
Similarly, the property of being $k$-common translates to the language of graph limits as follows.
A graph $H$ is $k$-common if 
\[t(H,W_1)+\cdots+t(H,W_k)\ge \frac{1}{k^{\|H\|-1}}\]
for any graphons $W_1,\ldots,W_k$ such that $W_1+\cdots+W_k=1$.

We pause the exposition of graph limit theory to demonstrate
how the just introduced notions are convenient for establishing some basic properties of $k$-common graphs.
Jagger, \v{S}\v{t}ov\'{\i}\v{c}ek and Thomason~\cite[Theorem~13]{JagST96} observed that
if $H$ is not $k$-common, then $H$ is not $\ell$-common for any $\ell\geq k$.
We now present their argument in the language of graph limits.
Suppose that $H$ is not $k$-common,
i.e., there exists graphons $W_1,\ldots,W_k$ such that $W_1+\cdots+W_k=1$ and
$t(H,W_1)+\cdots+t(H,W_k) < k^{-\|H\|+1}$.
Consider an integer $\ell>k$.
We set $W'_i=\frac{k}{\ell}W_i$ for $i\in [k]$ and $W'_i=1/\ell$ for $i\in [\ell]\setminus [k]$.
Observe that
\begin{align*}
t(H,W'_1)+\cdots+t(H,W'_{\ell}) &= \left(\frac{k}{\ell}\right)^{\|H\|}\left(t(H,W_1)+\cdots+t(H,W_k)\right)+\frac{\ell-k}{\ell^{\|H\|}}\\
                              &< \frac{k}{\ell^{\|H\|}}+\frac{\ell-k}{\ell^{\|H\|}}=\ell^{-\|H\|+1},
\end{align*}
which implies that $H$ is not $\ell$-common.
Hence, we can define $\kappa(H)$ to be the smallest integer $k$ such that $H$ is not $k$-common;
if no such integer exists, we set $\kappa(H)=\infty$.
That is, $H$ is $k$-common if and only if $2\leq k<\kappa(H)$.
In particular, Theorem~\ref{thm:Sidorenko} asserts that $H$ is Sidorenko if and only if $\kappa(H)=\infty$.

In Section~\ref{sec:intro},
for any non-bipartite connected graph $H$,
we exhibited a $k$-edge-coloring of $K_N$ from~\cite{JagST96} which
has $(1+o(1))N^{|H|}2^{-(k-1)(|H|-1)}$ monochromatic copies of $H$.
It follows that a non-bipartite connected graph $H$ is not $k$-common for any $k$
that satisfies $2^{-(k-1)(|H|-1)}<k^{-\|H\|+1}$.
This implies that
$\kappa(H)\le \lceil 2 d\log_2 d\rceil$
for any non-bipartite connected graph $H$ with average degree $d$.
We remark that for graphs $H$ with chromatic number larger than three,
a better upper bound on $\kappa(H)$ can be obtained by considering the edge-coloring obtained
by splitting vertices of $K_N$ to $(\chi(H)-1)^{k-1}$ roughly equal parts and
defining the edge-coloring based on the base $(\chi(H)-1)$ representations of the indices of the parts.

Let us return to our brief introduction to notions from the theory of graph limits that we use in this paper.
A graphon $W$ can be thought of as an operator on $L_2[0,1]$ where the image of a function $f\in L_2[0,1]$ is
given by
\[\int_{0}^1W(x,y)f(y)\dd y.\]
Every such operator is compact and so
its spectrum $\sigma(W)$ is either finite or countably infinite,
the only accumulation point of $\sigma(W)$ can be zero and
every non-zero element of $\sigma(W)$ is an eigenvalue of $W$~\cite[Section~7.5]{Lov12}.
In addition, all elements of $\sigma(W)$ are real and the largest is at least the density of $W$.
We define $\wsigma(W)$ to be the multiset containing all non-zero elements $\lambda$ of $\sigma(W)$,
with multiplicity equal to the dimension of the kernel of $(W-\lambda)$, which is finite.
In the graph case, the trace of the $n$-th power of the adjacency matrix of a graph $G$,
which is equal to the sum of the $n$-th powers of the eigenvalues of the matrix,
is the number of homomorphisms from $C_n$ to $G$,
i.e., it is equal to $t(C_n,G)|G|^n$~\cite[Equation~(5.31)]{Lov12}.
We will need the analogous statement for graphons,
which we now state as a proposition.
\begin{proposition}[{Lov\'asz~\cite[Equation~(7.22)]{Lov12}}]
\label{prop:spectrum}
Let $W$ be a graphon. It holds for every $n\ge 3$:
\[t(C_n,W)=\sum_{\lambda\in\wsigma(W)}\lambda^n.\]
\end{proposition}

There are several useful metrics on graphons. One of the most important
from the perspective of graph limit theory is the metric induced by the
cut norm. A \emph{kernel} is a bounded symmetric measurable function from 
$[0,1]^2$ to $\mathbb{R}$; a kernel can be thought of as a continuous 
variant of the adjacency matrix of an edge-weighted graph. We define the 
\emph{cut norm} of a kernel $U$ to be
\[\cut{U} := \sup_{S,T\subseteq [0,1]}\left|\int_{S\times T}U(x,y)\dd x\dd y\right|,\]
where the supremum is over all measurable subsets $S$ and $T$ of $[0,1]$. 
The \emph{cut distance} of graphons $W$ and $W'$, denoted by $\dcut{W}{W'}$,
is the infimum of the cut norm $\cut{W^\varphi-W'}$
taken over all measure preserving maps $\varphi:[0,1]\to [0,1]$
where $W^\varphi(x,y)=W(\varphi(x),\varphi(y))$.
If two graphons have small cut distance, then their
homomorphism densities do not differ substantially, as the next lemma shows.
\begin{lemma}[{Lov\'asz~\cite[Lemma 10.23]{Lov12}}]
\label{lm:cutdist}
Let $W$ and $W'$ be two graphons and $H$ a graph.
It holds that $|t(H,W)-t(H,W')|\le \|H\|\cdot \dcut{W}{W'}$.
\end{lemma}
Lemma~\ref{lm:cutdist} asserts that two graphons which are close in the cut distance 
have similar homomorphism densities. The next lemma allows us to find a step graphon 
of bounded complexity that is close in cut distance to any graphon.
\begin{lemma}[Frieze and Kannan~\cite{FriK99}; see also~{\cite[Lemma~9.3]{Lov12}}]
\label{lm:reg}
For every $\varepsilon>0$, there exists an integer $M\in\NN$ such that for every graphon $W$,
there exists a step graphon $W'$ with at most $M$ parts, all of equal sizes, such that
the densities of $W$ and $W'$ are the same and $\dcut{W}{W'}\le\varepsilon$.
\end{lemma}
The homomorphism density function extends naturally to kernels $U$ by setting $t(H,U)$ 
to be the integral in~\eqref{eq:tHW} with $W$ replaced by $U$.
A graphon $W$ that is close to the $p$-constant graphon
can be expressed as $p+\varepsilon U$ for some kernel $U$ and small $\varepsilon>0$.
The following proposition provides a useful expansion of $t(H,p+\varepsilon U)$,
which implicitly appeared in~\cite{Lov11,Sid89};
we use the formulation from~\cite[proof of Proposition 16.27]{Lov12}.
\begin{proposition}
\label{prop:epsU}
Let $U$ be a kernel, $H$ a graph and $p\in [0,1]$.
It holds that
\[t(H,p+\varepsilon U)=\sum_{F\subseteq E(H)}t(H[F],U)p^{\|H\|-|F|}\varepsilon^{|F|}\]
where $H[F]$ is the spanning subgraph of $H$ with the edge set $F$.
\end{proposition}

A local variant Sidorenko's Conjecture was considered in~\cite{Lov11} and in~\cite[Chapter 16]{Lov12}.
Here, we consider a stronger notion discussed in~\cite{FoxW17}:
a graph $H$ is \emph{locally Sidorenko} if there exists $\varepsilon_0>0$ such that
for every graphon $W$ with density $p$ such that $\cut{W-p}\le\varepsilon_0p$ and $\|W-p\|_{\infty}\le p$,
it holds that $t(H,W)\ge p^{\|H\|}$. The following theorem characterized locally
Sidorenko graphs. 
\begin{theorem}[Fox and Wei~\cite{FoxW17}]
\label{thm:localSid}
A graph $H$ is locally Sidorenko if and only if $H$ is forest or its girth is even.
\end{theorem}
Similarly, we say that a graph $H$ is \emph{locally $k$-common}
if for every $k\ge 2$, there exists $\varepsilon_0>0$ such that
\[t(H,W_1)+\cdots+t(H,W_k)\ge k^{-\|H\|+1}\]
for all graphons $W_1,\ldots,W_k$ such that $W_1+\cdots+W_k=1$, $\cut{W_i-1/k}\le\varepsilon_0/k$ 
and $\|W_i-1/k\|_{\infty}\le 1/k$ for all $i\in [k]$.

Fix a graphon $W$ and a real $\delta>0$ and
consider the set $\AAA(W,\delta)$ of all measurable functions $h:[0,1]\to [0,1]$ such that
\[\int_{[0,1]^2} h(x)W(x,y)h(y)\dd x\dd y\le\delta\|h\|_1^2.\]
Intuitively, for $\|h\|_1>0$, one can think of $h$ as a weight function on $[0,1]$ with the property that, 
if $x$ and $y$ are chosen independently at random according to the probability measure induced by
$h/\|h\|_1$, then the expected value of $W(x,y)$ is at most $\delta$. 
We define the \emph{$\delta$-independence ratio} of $W$ to be
\[\alpha_{\delta}(W) := \sup_{h\in\AAA(W,\delta)}\|h\|_1.\]

We next define a notion of a subgraphon that
is more involved than restricting a graphon to a measurable subset of $[0,1]$ and rescaling.
This notion will be used in the proof of Theorem~\ref{thm:append}
to apply induction to a ``sparse'' part of one of the graphons $W_1,W_2,\dots,W_k$.
Let $h:[0,1]\to [0,1]$ be a measurable function such that $\|h\|_1>0$ and
let $f:[0,\|h\|_1]\to [0,1]$ be the measurable function defined by
\[f(z) := \inf \left\{t\in [0,1]\mbox{ such that }\int_{[0,t]} h(x)\dd x\ge z\right\}.\]
Observe that
\[\int_A h(x)\dd x=|f^{-1}(A)|\]
for every measurable subset $A$ of $[0,1]$.
The subgraphon of $W$ induced by $h$, which is denoted by $W[h]$, is the graphon defined by
\[W[h](x,y):=W(f(x\cdot \|h\|_1),f(y\cdot \|h\|_1))\]
for every $(x,y)\in [0,1]^2$.
The graphon $W[h]$ is associated with the following sequence of random graphs.
Choose $n$ points independently at random based on the probability with density $h/\|h\|_1$ and
form a graph $G_{n}$ with vertex set $[n]$ by joining vertices $i$ and $j$ with probability $W(x_i,x_j)$.
Then $W[h]$ is a limit of the sequence $(G_{n})_{n\in\NN}$ with probability one. 
The definition of $W[h]$ implies that
\begin{equation}
t(H,W[h])=\frac{1}{\|h\|_1^{|H|}}\int_{[0,1]^{V(H)}} \prod_{u\in V(H)}h(u) \prod_{uv\in E(H)} W(x_u,x_v) \dd x_{V(H)}\label{eq:tWh}
\end{equation}
for every graph $H$.
In particular, $t(H,W)$ is at least $\|h\|_1^{|H|}\cdot t(H,W[h])$.

We conclude this section by relating certain ``reflection operations'' to homomorphism densities.
The arguments of this kind are standard in the area;
however, we have decided to provide a self-contained exposition for completeness.
Let $H$ be a graph and let $U\subseteq V(H)$ be an independent set of vertices of $H$.
For a graphon $W$, we define a function $t_W^H:[0,1]^U\to\RR$ as follows:
\[t_W^H(x_U)=\int_{[0,1]^{V(H)\setminus U}}\prod_{vv'\in E(H)}W(x_v,x_{v'})\dd x_{V(H)\setminus U};\]
note that the function $t_W^H$ depends on the choice of the set $U$.
Since the choice of the set $U$ will always be clear from the context,
we have decided not to include the set $U$ in the notation explicitly to keep the used notation simple.
Informally speaking, the function $t_W^H(x_U)$ counts the number of homomorphic copies of $H$ rooted at $x_U$.
Observe that
\[t(H,W)=\int_{[0,1]^U}t_W^H(x_U)\dd x_U.\]
We now state a proposition, which gives a lower bound on the homomorphism density of a graph obtained by reflecting $H$ along the set $U$.
\begin{proposition}
\label{prop:reflect}
Let $H$ be a graph, $n$ a positive integer and $U\subseteq V(H)$ an independent set of vertices of $H$.
Further,
let $H^n$ be the graph obtained by taking $n$ copies of $H$ and identifying the corresponding vertices of the set $U$,
i.e., the graph $H^n$ has $n|H|-(n-1)|U|$ vertices.
The following holds for every graphon $W$:
\[t(H^n,W)\ge t(H,W)^n.\]
\end{proposition}
\begin{proof}
Fix a graphon $W$.
We consider both graphs $H$ and $H^n$ with the set $U$ and note that
$t_W^{H^n}(x_U)=t_W^H(x_U)^n$ for every $x_U\in [0,1]^U$.
Hence, it follows that
\[t(H^n,W)=\int_{[0,1]^U}t_W^{H^n}(x_U)\dd x_U\ge\left(\int_{[0,1]^U}t_W^H(x_U)\dd x_U\right)^n=t(H,W)^n\]
by Jensen's Inequality.
\end{proof}
The same argument translates to the rooted setting,
which we formulate here for future reference but omit the proof as it is completely analogous to the proof of Proposition~\ref{prop:reflect}.
\begin{proposition}
\label{prop:reflect-rooted}
Let $H$ be a graph, $n$ a positive integer, $U\subseteq V(H)$ an independent set of vertices of $H$, and
$U'\subseteq V(H)$ an independent set that is a superset of $U$.
Further, let $H^n$ be the graph obtained from $H$ taking $n$ copies of $H$ and identifying the corresponding vertices of the set $U'$.
The following holds for every graphon $W$ and every $x_U\in [0,1]^U$:
\[t_W^{H^n}(x_U)\ge t_W^H(x_U)^n.\]
\end{proposition}
The following proposition is obtained by two applications of Proposition~\ref{prop:reflect},
first to the graph $K_{2,2}$ and $U$ being one of the two parts of $K_{2,2}$, and
second to the graph $K_{2,2n}$ and $U$ being the $2n$-vertex part of $K_{2,2n}$.
\begin{proposition}
\label{prop:Knn}
The following holds for every graphon $W$ and every $n\in\NN$:
\[t(K_{2n,2n},W)\ge t(K_{2,2},W)^{n^2}.\]
\end{proposition}

\section{Non-bipartite $k$-common graphs}
\label{sec:conn}

This section is devoted to the proof of Theorem~\ref{thm:append}.
For $a,b\geq1$, we let $K_{2a,2b,C_5}$ be the graph obtained from $K_{2a,2b}$ by adding 
$b$ disjoint copies of $C_5$ and identifying one vertex of each of these copies with one vertex in the $2b$-vertex part of $K_{2a,2b}$ (each copy involves a different vertex of the part).
In particular, $K_{2n,2n,C_5}$ is the graph from the statement of Theorem~\ref{thm:append}.
We start with proving that $K_{2n,2,C_5}$ is locally Sidorenko in a certain strong sense;
note that the assumption on $W$ is weaker than that in the local Sidorenko property discussed in Section~\ref{sec:prelim}
since we do not require any bound on  $\|W-p\|_{\infty}$.

\begin{lemma}
\label{lm:Kn2-local}
For every $p_0\in(0,1)$, there exist $\varepsilon_0\in(0,1)$ such that the following holds.
If $W$ is a graphon with density $p\ge p_0$ such that $t(K_{2,2},W)\le p^4+\varepsilon_0$,
then $t(K_{2n,2,C_5},W)\ge p^{4n+5}$ for all $n\in\NN$.
\end{lemma}

\begin{proof}
We show that the statement of the lemma holds for $\varepsilon_0=p_0^7/16$.
Throughout the proof, fix a graphon $W$ with density $p\ge p_0$ such that $t(K_{2,2},W)-p^4=\varepsilon\le\varepsilon_0$.
If the set $\wsigma(W)$ is finite, then set $I=[|\wsigma(W)|]$ and set $I=\NN$ otherwise.
Let $\lambda_i$, $i\in I$, be the elements of $\wsigma(W)$ listed in the decreasing order of their absolute value.
Further, let $g_i:[0,1]\to \RR$ be an eigenfunction corresponding to $\lambda_i$.
Without loss of generality,
we assume that $\|g_i\|_2=1$ for all $i\in I$ and that
the eigenfunctions are orthogonal to one another, i.e.,
\[\int_{[0,1]} g_i(x)g_{i'}(x)\dd x=0\]
for any two distinct $i$ and $i'$ from $I$.
In particular, the functions $G_i$, $i\in I$, are orthonormal.
Since it holds that
\[\lambda_1=\max_{\substack{f\in L_2[0,1]\\ \|f\|_2=1}}\int_{[0,1]^2} f(x)W(x,y)f(y)\dd x\dd y,\]
it follows $\lambda_1\ge p$. In particular, $\lambda_1\ge p_0$.

For every $x\in [0,1]$, we define a measurable function $f_x:[0,1]\to [0,1]$ by setting $f_x(y)=W(x,y)$ for all $y\in[0,1]$,
i.e., $f_x$ describes the ``neighborhood'' of $x$ in the graphon $W$.
We next define functions $\beta_i$ such that
$\beta_i(x)$ would be the coordinate of $f_x$ with respect to $g_i$, $i\in I$, for an orthonormal basis 
extending $g_i$, $i\in I$,
i.e.,
\[\beta_i(x)=\int_{[0,1]} g_i(y)f_x(y)\dd y.\]
Since the $L_2$-norm of $f_x$ is at most one and the functions $g_i$, $i\in I$, are orthonormal,
we obtain that
\begin{equation}
\sum_{i\in I}\beta_i(x)^2\le 1\label{eq:norm}
\end{equation}
for every $x\in [0,1]$.
Since the series
\[\sum_{i\in I}\lambda_i g_i(x)g_i(y)\]
converges to $W$ in the $L_2$-norm~\cite[Section 7.5]{Lov12}
it follows that $\|\beta_i-\lambda_i g_i\|_2=0$,
i.e., $\beta_i(x)=\lambda_i g_i(x)$ for almost every $x\in [0,1]$.
In particular, it holds that
\begin{equation}
\label{eq:alpha1lambda1}
\int_{[0,1]}\beta_i(x)^2\dd x=\lambda_i^2.
\end{equation}

Next consider a cycle $C_k$ and let $U$ consist of any single vertex of $C_k$.
As the functions $g_i$, $i\in I$, are orthonormal and are eigenfunctions of $W$,
we get that
\begin{align}
t_W^{C_k}(x) &=\int_{[0,1]^{k-1}}f_x(y_1)W(y_1,y_2)W(y_2,y_3)\cdots W(y_{k-2},y_{k-1})f_x(y_{k-1})\dd y_1\cdots y_{k-1}\nonumber\\
             &=\sum_{i\in I}\lambda_i^{k-2}\beta_i(x)^2\label{eq:Ckx}
\end{align}
holds for every $k\ge 3$ and $x\in [0,1]$.
It follows that
\begin{equation}t(C_k,W)=\int_{[0,1]}t_W^{C_k}(x)\dd x=\sum_{i\in I}\lambda_i^{k-2}\int_{[0,1]}\beta_i(x)^2\dd x.\label{eq:Ckalpha}\end{equation}
On the other hand, Proposition~\ref{prop:spectrum} tells us that
\begin{equation}
t(C_{k},W)=\sum_{i\in I}\lambda_i^{k}.\label{eq:spectrum}
\end{equation}
In particular, we obtain for $k=4$ that
\[\varepsilon=t(K_{2,2},W)-p^4=\sum_{i\in I}\lambda_i^4-p^4\ge\sum_{i\in I\setminus\{1\}}\lambda_i^4,\]
which implies that $|\lambda_i|\le\varepsilon^{1/4}$ for every $i\in I\setminus\{1\}$.
In particular, $\lambda_1$ has multiplicity one.

Our aim is to estimate $t_W^{K_{2n,2,C_5}}(x)$
where $U$ is the set consisting of the vertex shared by $K_{2n,2}$ and $C_5$.
Observe that 
\begin{equation}
t_W^{K_{2n,2,C_5}}(x)=t_W^{K_{2n,2}}(x)\cdot t_W^{C_5}(x).\label{eq:prodKC5}
\end{equation}
We start by rewriting the identity \eqref{eq:Ckx} for $k=4$ and $k=5$:
\begin{align}
t_W^{C_4}(x) & =\lambda_1^2\beta_1^2(x)+\sum_{i\in I\setminus\{1\}}\lambda_i^2\beta_i^2(x) \label{eq:C4}\\
t_W^{C_5}(x) & =\lambda_1^3\beta_1^2(x)+\sum_{i\in I\setminus\{1\}}\lambda_i^3\beta_i^2(x). \label{eq:C5}
\end{align}
Note that all of the terms on the right sides of these two expressions are non-negative,
except for possibly the summation in \eqref{eq:C5}.
Using Proposition~\ref{prop:reflect-rooted} and the equation~\eqref{eq:C4},
we obtain that
\begin{align}
t_W^{K_{2n,2}}(x) &\ge t_W^{C_4}(x)^n\nonumber\\
                  & = \left(\lambda_1^2\beta_1^2(x)+\sum_{i\in I\setminus\{1\}}\lambda_i^2\beta_i^2(x)\right)^n\nonumber\\
		  & \ge\lambda_1^{2n}\beta_1^{2n}(x)+\lambda_1^{2n-2}\beta_1^{2n-2}(x)\sum_{i\in I\setminus\{1\}}\lambda_i^2\beta_i^2(x)\label{eq:C4n}
\end{align}
Our next goal is to show that, unless $f_x$ is almost completely orthogonal to $g_1$,
the homomorphism density of $K_{2n,2,C_5}$ rooted at $x$ is at least its expected average value.
Specifically, we will set $\pi_0=p_0^2/2$ and show that if $\beta_1^2(x)\ge\pi_0$, then
\begin{equation}
t_W^{K_{2n,2,C_5}}(x)\ge \lambda_1^{2n+3}\beta_1^{2n+2}(x).\label{eq:estimKC5}
\end{equation}
To this end, we substitute \eqref{eq:C5} and \eqref{eq:C4n} into \eqref{eq:prodKC5} to obtain
\begin{align*}
t_W^{K_{2n,2,C_5}}(x) & \ge\left(\lambda_1^{2n}\beta_1^{2n}(x)+\lambda_1^{2n-2}\beta_1^{2n-2}(x)\sum_{i\in I\setminus\{1\}}\lambda_i^2\beta_i^2(x)\right) \\
                      & \times\left(\lambda_1^3\beta_1^2(x)+\sum_{i\in I\setminus\{1\}}\lambda_i^3\beta_i^2(x)\right).
\end{align*}		      
Multiplying out, we obtain four terms.
One of them is the right side of \eqref{eq:estimKC5} and the remaining three terms are as follows:
\begin{align*}
& \lambda_1^{2n+1}\beta_1^{2n}(x)\left(\sum_{i\in I\setminus\{1\}}\lambda_i^2\beta_i^2(x)\right),\\
& \lambda_1^{2n}\beta_1^{2n}(x)\left(\sum_{i\in I\setminus\{1\}}\lambda_i^3\beta_i^2(x)\right) \mbox{ and}\\
& \lambda_1^{2n-2}\beta_1^{2n-2}(x)\left(\sum_{i\in I\setminus\{1\}}\lambda_i^2\beta_i^2(x)\right)\left(\sum_{i\in I\setminus\{1\}}\lambda_i^3\beta_i^2(x)\right).
\end{align*}
So, to establish \eqref{eq:estimKC5}, we need to show that the sum of these three terms is non-negative.
We first consider the sum of half of the first term and the whole of the second term.
Since $p_0\leq \lambda_1$ and $\lambda_i\leq \varepsilon^{1/4}$ for all $i\in I\setminus \{1\}$, we get
\begin{align*}
& \frac{1}{2}\lambda_1^{2n+1}\beta_1^{2n}(x)\left(\sum_{i\in I\setminus\{1\}}\lambda_i^2\beta_i^2(x)\right)+
  \lambda_1^{2n}\beta_1^{2n}(x)\left(\sum_{i\in I\setminus\{1\}}\lambda_i^3\beta_i^2(x)\right)\\
\ge &
  \left(\frac{p_0}{2}-\varepsilon^{1/4}\right)\left(\lambda_1^{2n}\beta_1^{2n}(x)\right)\left(\sum_{i\in I\setminus\{1\}}\lambda_i^2\beta_i^2(x)\right)\ge 0.
\end{align*}
Next, we estimate the sum of half of the first term and the third term as follows:
\begin{align*}
& \frac{1}{2}\lambda_1^{2n+1}\beta_1^{2n}(x)\left(\sum_{i\in I\setminus\{1\}}\lambda_i^2\beta_i^2(x)\right)\\
& +
  \lambda_1^{2n-2}\beta_1^{2n-2}(x)\left(\sum_{i\in I\setminus\{1\}}\lambda_i^2\beta_i^2(x)\right)\left(\sum_{i\in I\setminus\{1\}}\lambda_i^3\beta_i^2(x)\right)\\
\ge &
  \left(\frac{1}{2}\lambda_1^3\beta_1^2(x)-\sum_{i\in I\setminus\{1\}}|\lambda_i|^3\beta_i^2(x)\right)
  \lambda_1^{2n-2}\beta_1^{2n-2}(x)
  \left(\sum_{i\in I\setminus\{1\}}\lambda_i^2\beta_i^2(x)\right)\\
\ge &
  \left(\frac{p^3\pi_0}{2}-\varepsilon^{3/4}\sum_{i\in I\setminus\{1\}}\beta_i^2(x)\right)
  \lambda_1^{2n-2}\beta_1^{2n-2}(x)
  \left(\sum_{i\in I\setminus\{1\}}\lambda_i^2\beta_i^2(x)\right)\\
\ge &
  \left(\frac{p_0^5}{4}-\varepsilon^{3/4}\right)
  \lambda_1^{2n-2}\beta_1^{2n-2}(x)
  \left(\sum_{i\in I\setminus\{1\}}\lambda_i^2\beta_i^2(x)\right).
\end{align*}  
The last inequality follows from \eqref{eq:norm}. The final expression is non-negative (with room to spare)
by the choice of $\varepsilon_0$. 

The statement would follow from \eqref{eq:alpha1lambda1} and \eqref{eq:estimKC5} by a convexity argument
if $\beta_1^2(x)\ge\pi_0$ held for almost all $x\in [0,1]$.
As this need not be the case for almost all $x\in [0,1]$, a finer argument is needed.
Let $X_1$ be the set of $x\in[0,1]$ such that $\beta_1^2(x)\ge\pi_0$ and let $\delta=1-|X_1|$.
By \eqref{eq:alpha1lambda1} for $i=1$, we have
\[\int_{X_1}\beta_1^2(x)\dd x=\int_{[0,1]}\beta_1^2(x)\dd x-\int_{[0,1]\setminus X_1}\beta_1^2(x)\dd x\ge \lambda_1^2-\delta\pi_0.\]
The equation \eqref{eq:alpha1lambda1} for $i=1$ also implies that $\delta<1$;
otherwise, the integral of $\beta_1^2(x)$, which is equal to $\lambda_1^2\ge p_0^2$, would be at most $\pi_0=p_0^2/2$.
Using Jensen's Inequality, we have
\begin{align*}
\int_{X_1}\beta_1^{2n+2}(x)\dd x & \ge \frac{(\lambda_1^2-\delta\pi_0)^{n+1}}{(1-\delta)^n}\\
  & = \left(\frac{\lambda_1^2-\delta\pi_0}{1-\delta}\right)^{n-1}\cdot\frac{\lambda_1^4-2\delta\pi_0\lambda_1^2+\delta^2\pi_0^2}{1-\delta} \\
  & \ge \lambda_1^{2n-2}\cdot\frac{\lambda_1^4-2\delta\pi_0\lambda_1^2+\delta^2\pi_0^2}{1-\delta} \\
  & = \lambda_1^{2n-2}\cdot\left(\lambda_1^4+\frac{\delta\lambda_1^4-2\delta\pi_0\lambda_1^2+\delta^2\pi_0^2}{1-\delta}\right) \\
  & \ge \lambda_1^{2n+2}+\lambda_1^{2n-2}\cdot\frac{\delta\lambda_1^2(\lambda_1^2-2\pi_0)}{1-\delta}\ge \lambda_1^{2n+2}.
\end{align*}
In the step between the second and third lines and in the last line,
we used the fact that $2\pi_0=p_0^2\leq p^2\leq \lambda_1^2$. 
Since the estimate \eqref{eq:estimKC5} holds for every $x\in X_1$, we obtain that
\[
t(K_{2n,2,C_5},W) \ge\int_{X_1}t_W^{K_{2n,2,C_5}}(x)\dd x
		 \ge\int_{X_1}\lambda_1^{2n+3}\beta_1^{2n+2}(x)\dd x \ge \lambda_1^{4n+5}\ge p^{4n+5}.
\]
This concludes the proof of the lemma.
\end{proof}

The next lemma follows from Lemma~\ref{lm:Kn2-local} by applying Proposition~\ref{prop:reflect-rooted}
for the graph $H=K_{2n,2,C_5}$ and the set $U$ being the part of $K_{2n,2}$ with $2n$ vertices.

\begin{lemma}
\label{lm:Knn-local}
For every $p_0\in(0,1)$, there exists $\varepsilon_0\in (0,1)$ such that the following holds.
If $W$ is a graphon with density $p\ge p_0$ such that $t(K_{2,2},W)\le p^4+\varepsilon_0$,
then $t(K_{2n,2n,C_5},W))\ge p^{4n^2+5n}$ for all $n\in\NN$.
\end{lemma}

The second ingredient for the proof of Theorem~\ref{thm:append} is the next lemma,
which covers the case when $t(K_{2,2},W)$ is substantially larger than $t(K_2,W)^4$
unless the graphon $W$ contains a large sparse part.

\begin{lemma}
\label{lm:Knn-far}
For every $p_0\in (0,1)$ and every $\varepsilon_0\in (0,1)$,
there exist $n_0\in\NN$ and $\delta_0\in (0,1)$ such that
the following holds for every graphon $W$ with density $p\ge p_0$ such that $t(K_{2,2},W)\ge p^4+\varepsilon_0$:
\begin{itemize}
\item $t(K_{2n,2n,C_5},W)\ge p^{4n^2+5n}$ for every $n\ge n_0$, or
\item $\alpha_{p_0}(W)\ge\delta_0$.
\end{itemize}
\end{lemma}

\begin{proof}
Set $\delta_0:=p_0\varepsilon_0/16$ and set $d_0:=\delta_0$.
The reason that we let $\delta_0$ and $d_0$ to represent the same quantity is that
they play different roles in the proof; $\delta_0$ is the lower bound on the 
$p_0$-independence ratio in the statement of the theorem whereas $d_0$ is the 
threshold for considering a point $x\in [0,1]$ to have ``small degree'' in a 
graphon $W$. Choose $n_0$ to be large enough so that
\[\left(1+\varepsilon_0/2\right)^{n_0}d_0^4p_0^3\ge 1.\]

Fix a graphon $W$ with density $p\ge p_0$ such that $t(K_{2,2},W)\ge p^4+\varepsilon_0$.
We iteratively define sets $A_i$, $i\in\NN$, such that
$A_i$ is the set of all $x\in [0,1]$ with ``small degree'' when disregarding neighbors in $A_{i-1}$.
Formally, we let $A_0=\emptyset$ and let
 $A_i$, $i\in\NN$, be the set of all $x\in [0,1]$ such that
\[\int_{[0,1]\setminus A_{i-1}}W(x,y)\dd y\le d_0.\]
Note that $A_{i-1}\subseteq A_i$ for every $i\in\NN$.
Let $A$ be the union of all sets $A_i$, $i\in\NN$, and
observe that, for every $x\in [0,1]\setminus A$,
\[\int_{[0,1]\setminus A} W(x,y)\dd y=\lim_{i\to\infty}\int_{[0,1]\setminus A_{i-1}}W(x,y)\dd y.\]
In particular, it holds that
\[\int_{[0,1]\setminus A} W(x,y)\dd y\ge d_0\]
for every $x\in [0,1]\setminus A$.

We next distinguish two cases depending on the measure of $A$, and
we first analyze the case that $|A|\ge\varepsilon_0/8$.
We start with estimating the density of $W$ on the set $A$:
\begin{align*}
\int_{A^2} W(x,y)\dd x\dd y &= \sum_{i\in\NN}\int\limits_{(A_i\setminus A_{i-1})^2}W(x,y)\dd x\dd y 
                              + 2\quad\int\limits_{\mathclap{(A_i\setminus A_{i-1})\times(A\setminus A_{i})}}\quad W(x,y)\dd x\dd y\\
                            & \le \sum_{i\in\NN}\int\limits_{(A_i\setminus A_{i-1})^2}W(x,y)\dd x\dd y 
			      +2\quad\int\limits_{\mathclap{(A_i\setminus A_{i-1})\times([0,1]\setminus A_i)}}\quad W(x,y)\dd x\dd y\\
                            & \le 2\sum_{i\in\NN}\int\limits_{(A_i\setminus A_{i-1})\times([0,1]\setminus A_{i-1})}W(x,y)\dd x\dd y\\
			    & \le 2\sum_{i\in\NN}\left|A_i\setminus A_{i-1}\right|d_0\le 2|A|d_0.
\end{align*}
It follows that
\[\frac{\int_{A^2} W(x,y)\dd x\dd y}{|A|^2}\le\frac{2d_0}{|A|}=\frac{p_0\varepsilon_0}{8|A|}\le p_0.\]
Thus, the characteristic function of $A$ certifies that $\alpha_{p_0}(W)\ge\varepsilon_0/8\ge\delta_0$.

In the rest of the proof, we assume that $|A|\le\varepsilon_0/8$.
We show that the homomorphism density of $K_{2n,2n,C_5}$ is large enough even if we disregard the points contained in $A$.
To do this, we set $W'$ to be the graphon defined by
\[W'(x,y)=\begin{cases}
          0 & \mbox{if $x\in A$ or $y\in A$,}\\
          W(x,y) & \mbox{otherwise.}
	  \end{cases}\]
We next estimate the homomorphism density $K_{2n,2n}$ in $W'$ using Proposition~\ref{prop:Knn} as follows:
\begin{align*}
t(K_{2n,2n},W') & \ge t(K_{2,2},W')^{n^2}\\
                & \ge \left(t(K_{2,2},W)-4|A|\right)^{n^2}\\
		& \ge \left(p^4+\varepsilon_0-\varepsilon_0/2\right)^{n^2}=\left(p^4+\varepsilon_0/2\right)^{n^2}
\end{align*}
We next combine these copies of $K_{2n,2n}$ with copies of $C_5$ rooted at $x\in [0,1]\setminus A$
unless $W$ contains a sparse part.
Consider $x\in [0,1]\setminus A$ and let $h(y)=W'(x,y)$.
Note that
\[\int_{[0,1]}h(y)\dd y=\int_{[0,1]}W'(x,y)\dd y=\int_{[0,1]\setminus A}W(x,y)\ge d_0=\delta_0.\]
Since $h(y)=0$ for $y\in A$, we obtain that
\begin{equation}
\int_{[0,1]^2}h(y)W'(y,z)h(z)\dd y\dd z=\int_{[0,1]^2}h(y)W(y,z)h(z)\dd y\dd z.\label{eq:hdens}
\end{equation}
If the integral in~\eqref{eq:hdens} is less than $p_0\|h\|_1^2$,
then $\alpha_{p_0}(W)\ge\delta_0$,
which is the second conclusion of the lemma.

Hence, we can assume that the integral in~\eqref{eq:hdens} is at least $p_0\|h\|_1^2$ for every $x\in [0,1]\setminus A$.
Since the $3$-edge path $P_4$ is Sidorenko, we conclude by considering the graphon $W[h]$ that
\[t_{W'}^{C_5}(x)\ge \|h\|_1^4\cdot t(P_4,W[h])\ge \|h\|_1^4 p_0^3\ge d_0^4p_0^3\]
for every $x\in [0,1]\setminus A$.
It follows that
\begin{align*}
t(K_{2n,2n,C_5},W) & \ge t(K_{2n,2n,C_5},W') \\
                   & \ge t(K_{2n,2n},W')\cdot\left(d_0^4p_0^3\right)^n \\
		   & \ge \left(p^4+\varepsilon_0/2\right)^{n^2}\left(d_0^4p_0^3\right)^n \\
		   & \ge p^{4n^2}\left(1+\varepsilon_0/2\right)^{n^2}\left(d_0^4p_0^3\right)^n \\
		   & \ge p^{4n^2}\left(\left(1+\varepsilon_0/2\right)^{n_0}d_0^4p_0^3\right)^n\ge p^{4n^2}\ge p^{4n^2+5n}.
\end{align*}
Hence, the first conclusion of the lemma holds.
\end{proof}

We are now ready to prove the main theorem of this section, which implies Theorem~\ref{thm:append}.
Theorem~\ref{thm:Knn-common} is a variant of Theorem~\ref{thm:append}
where a very small proportion of the edges can be left uncolored.
This additional flexibility is needed for an inductive argument used in the proof of the theorem.

\begin{theorem}
\label{thm:Knn-common}
For every $k\in\NN$, there exist $n_k\in\NN$ and $\delta_k\in (0,1)$ with the following property.
If $W_1,\ldots,W_k$ are graphons such that $t(K_2,W_1+\cdots+W_k)\ge 1-\delta_k$, then
\[\sum_{i\in [k]} t(K_{2n,2n,C_5},W_i)\ge\frac{t(K_2,W_1+\cdots+W_k)^{4n^2+5n}}{k^{4n^2+5n-1}}\]
for every $n\ge n_k$.
\end{theorem}

\begin{proof}
We proceed by induction on $k\in\NN$.
Suppose first that $k=1$.
We apply Lemma~\ref{lm:Knn-local} with $p_0=3/4$ to get $\varepsilon_0\in (0,1)$.
We show that the statement of the theorem is true for $n_1=1$ and $\delta_1=\varepsilon_0/4$.
Let $W_1$ be a graphon with density $p\ge 1-\delta_1\ge 3/4$.
Observe that
\[t(K_{2,2},W)-p^4\le 1-p^4\le 1-(1-\delta_1)^4\le 4\delta_1=\varepsilon_0.\]
Hence, Lemma~\ref{lm:Knn-local} implies that
\[t(K_{2n,2n,C_5},W_1)\ge p^{4n^2+5n}.\]
This completes the proof in the base case $k=1$.

Now, suppose that we have already established the existence of $n_1,\ldots,n_{k-1}$ and $\delta_1,\ldots,\delta_{k-1}$.
Choose $p_0=\delta_{k-1}/4k$ and apply Lemma~\ref{lm:Knn-local} to get $\varepsilon_0$.
We then apply Lemma~\ref{lm:Knn-far} with $p_0$ and $\varepsilon_0$ to obtain $n_0$ and $\delta_0$.
Set $\delta_k=\frac{\delta_{k-1}\delta_0^2}{4k}$.
Finally, choose $n_k$ such that $n_k\ge \max\{n_0, n_{k-1}\}$ and
\[\left(\frac{1}{k}+\frac{1}{2k(k-1)}\right)^{4n_k+5}\delta_0^{8}\ge\frac{k}{k-1}\left(\frac{1}{k}\right)^{4n_k+5}.\]
The choice of $n_k$ yields that the following holds for all $n\ge n_k$:
\[(k-1)\left(\frac{1}{k}+\frac{1}{2k(k-1)}\right)^{4n^2+5n}\delta_0^{8n}\ge k\left(\frac{1}{k}\right)^{4n^2+5n}.\]

Let graphons $W_1,\ldots,W_k$ satisfying the assumption of the theorem be given and let $n\ge n_k$.
Further, let $p=t(K_2,W_1+\cdots+W_k)$ be the density of the graphon $W_1+\cdots+W_k$; note that $p\ge 1-\delta_k$.

We distinguish two cases.
First suppose that there exists $i\in [k]$ such that $\alpha_{p_0}(W_i)\ge\delta_0$,
i.e., one of the graphons $W_1,\ldots,W_k$ contains a large sparse part.
Note that this case includes the case that the density of one of the graphons is at most $p_0$.
By symmetry, we can assume that $\alpha_{p_0}(W_k)\ge\delta_0$.
Let $h:[0,1]\to [0,1]$ be such that $\|h\|_1\ge\delta_0$ and
\[\int_{[0,1]^2}h(x)W_k(x,y)h(y)\dd x\dd y\le p_0\|h\|_1^2.\]
Since it holds that
\[\sum_{i\in [k]}\int_{[0,1]^2}h(x)W_i(x,y)h(y)\dd x\dd y\ge \|h\|_1^2-\delta_k,\]
we obtain that
\begin{align*}
\sum_{i\in [k-1]}\int_{[0,1]^2}h(x)W_i(x,y)h(y)\dd x\dd y & \ge \|h\|_1^2-p_0\|h\|_1^2-\delta_k\\
                                           & \ge \|h\|_1^2\left(1-\frac{\delta_{k-1}}{2k}\right)\\
		 		 	   & \ge \|h\|_1^2\left(1-\frac{1}{2k}\right).
\end{align*}
Since it holds that
\[t(K_2,W_1[h]+\cdots+W_{k-1}[h])\ge 1-\frac{\delta_{k-1}}{2k}\ge 1-\delta_{k-1},\]
we can apply induction to $W_1[h],\ldots,W_{k-1}[h]$ and arrive at the following:
\begin{align*}
\sum_{i\in [k-1]} t(K_{2n,2n,C_5},W_i) & \ge \|h\|_1^{8n}\sum_{i\in [k-1]}t(K_{2n,2n,C_5},W_i[h]) \\
     & \ge \|h\|_1^{8n}(k-1)\left(\frac{1-1/2k}{k-1}\right)^{4n^2+5n}\\
     & \ge \delta_0^{8n}(k-1)\left(\frac{1}{k}+\frac{1}{2k(k-1)}\right)^{4n^2+5n}\\
     & \ge k\left(\frac{1}{k}\right)^{4n^2+5n}\ge\frac{p^{4n^2+5n}}{k^{4n^2+5n-1}}.
\end{align*}
Hence, in the following, we assume that $\alpha_{p_0}(W_i)<\delta_0$ for every $i\in [k]$. 
In particular, we assume that $t(K_2,W_i)\ge p_0$ for every $i\in [k]$
and so we can apply Lemmas~\ref{lm:Knn-local} and~\ref{lm:Knn-far} to each of $W_1,\dots,W_k$.

Based on whether it holds that $t(K_{2,2},W_i)\le t(K_2,W_i)^4+\varepsilon_0$ or not,
Lemma~\ref{lm:Knn-local} or Lemma~\ref{lm:Knn-far}, respectively, implies
\[t(K_{2n,2n,C_5},W_i)\ge t(K_2,W_i)^{4n^2+5n}\]
for every $i\in [k]$.
Therefore, we obtain that
\[\sum_{i\in [k]}t(K_{2n,2n,C_5},W_i)\ge \sum_{i\in [k]} t(K_2,W_i)^{4n^2+5n}\]
which is at least $k\left(\frac{p}{k}\right)^{4n^2+5n}$ by convexity.
This concludes the proof of the theorem.
\end{proof}

\section{Sidorenko and locally Sidorenko graphs}
\label{sec:Sidorenko}

In this section, we prove that a graph is $k$-common for all $k\geq2$ if and only if
it is Sidorenko and that
no graph of odd girth is locally $k$-common for any $k\ge 3$.
We start with the former statement.

\begin{proof}[Proof of Theorem~\ref{thm:Sidorenko}]
We first show that if a graph $H$ is Sidorenko, then it is $k$-common for every $k\in\NN$.
Fix a Sidorenko graph $H$ and an integer $k\ge 2$.
Let $W_1,\ldots,W_k$ be graphons such that $W_1+\cdots+W_k=1$ and
let $p_1,\ldots,p_k$ be their respective densities.
Note that $p_1+\cdots+p_k=1$.
Since $H$ is Sidorenko,
\[t(H,W_1)+\cdots+t(H,W_k)\ge p_1^{\|H\|}+\cdots+p_k^{\|H\|}\ge k\left(\frac{p_1+\cdots+p_k}{k}\right)^{\|H\|}=k^{-\|H\|+1}.\]
Therefore, $H$ is $k$-common.

To complete the proof, we need to show that
if a graph $H$ is not Sidorenko, then there exists $k\ge 2$ such that $H$ is not $k$-common.
Fix a graph $H$ that is not Sidorenko and
let $W$ be a graphon with density $p$ such that $t(H,W)<p^{\|H\|}$.
Set $\varepsilon=p^{\|H\|}-t(H,W)$.
By Lemma~\ref{lm:reg}, there exists a step graphon $W'$ with density $p$ such that
the cut distance between $W$ and $W'$ is at most $\varepsilon/(2\|H\|)$.
Lemma~\ref{lm:cutdist} implies that
\[t(H,W')\le t(H,W)+\varepsilon/2=p^{\|H\|}-\varepsilon/2.\]
By splitting each of the parts of $W'$ into the same number of equal size smaller parts,
we can assume that the number $m$ of parts of $W'$ satisfies
\[4\|H\|\le m\varepsilon\qquad\mbox{and}\qquad p^{\|H\|}-\varepsilon/4<(p-1/m)^{\|H\|}.\]
Let $A_1,\ldots,A_m$ be the parts of $W'$ and
let $d_{ij}$, $i,j\in [m]$ be the value of $W'$ on the tile $A_i\times A_j$.
Further, let $\delta$ be the average of $d_{ij}$ taken over all pairs $i$ and $j$ such that $1\le i<j\le m$ and
let $W''$ be the step graphon with the same $m$ parts as $W'$ obtained from $W'$
by making each of the $m$ diagonal tiles to be equal to $\delta$.
Note that the density of the whole graphon $W''$ is $\delta$ and $\delta\ge p-1/m$.
Since the cut distance between $W'$ and $W''$ is at most $m/m^2=1/m$,
Lemma~\ref{lm:cutdist} implies that
\begin{align*}
t(H,W'') & \le t(H,W')+\|H\|\cdot\dcut{W}{W'}\\
        & \le t(H,W')+\varepsilon/4\\
	& \le p^{\|H\|}-\varepsilon/4<(p-1/m)^{\|H\|}\le\delta^{\|H\|}.
\end{align*}
Next choose an integer $\ell\in\NN$ such that $1\le\delta\ell m!$ and set $k=\ell m!$.
We next define $k$ graphons that witness that $H$ is not $k$-common;
the $k$ graphons will be indexed by pairs consisting of a permutation $\sigma\in S_m$ of order $m$ and an integer $s\in [\ell]$.
The graphon $W_{\sigma,s}$ for $\sigma\in S_m$ and $s\in [\ell]$
is the step graphon with $m$ parts $A_1,\ldots,A_m$ such that
the graphon $W_{\sigma,s}$ on a tile $A_i\times A_j$, $i,j\in [m]$,
is equal to $1/k$ if $i=j$ and
is equal to $\frac{d_{\sigma(i)\sigma(j)}}{k\delta}$ if $i\neq j$ (note that $\frac{d_{\sigma(i)\sigma(j)}}{k\delta}\le\frac{1}{k\delta}\le 1$).
Note that the density of each of the graphons $W_{\sigma,s}$ is $\frac{1}{k}$.
Moreover, the average value of all the $k$ graphons on any of the tiles is $\frac{1}{k}$.
Consequently, the $k$ graphons $W_{\sigma,s}$, $\sigma\in S_m$ and $s\in [\ell]$, sum to the $1$-constant graphon.
Since the homomorphism density of $H$ in each of the graphons $W_{\sigma,s}$, $\sigma\in S_m$ and $s\in [\ell]$,
is equal to $\frac{1}{(k\delta)^{\|H\|}}t(H,W'')<k^{-\|H\|}$,
it follows that $H$ is not $k$-common.
\end{proof}

We next show that locally $k$-common graphs for any $k\ge 3$
are precisely locally Sidorenko graphs (cf.~Theorem~\ref{thm:localSid}).

\begin{proof}[Proof of Theorem~\ref{thm:local}]
Fix an integer $k\ge 3$ for the proof, and
a graph $H$ with girth $\ell$ where $\ell$ is odd.

\begin{figure}
\begin{center}
\epsfbox{kcommon-2.mps}
\end{center}
\caption{The kernel $U$ used in the proof of Theorem~\ref{thm:local} for $\ell=5$.
         The origin of the coordinate system is in the top left corner.}
\label{fig:local}
\end{figure}

Let $A_1,\ldots,A_{2\ell}$ be any partition of the interval $[0,1]$ to $2\ell$ disjoint measurable sets,
each of measure $(2\ell)^{-1}$.
Consider a kernel $U$ defined as follows (also see Figure~\ref{fig:local}):
\[U(x,y)=\begin{cases}
         +1 & \mbox{if $x\in A_i$, $y\in A_j$, $\lceil i/\ell\rceil=\lceil j/\ell\rceil$ and $i=(j\pm 1)\bmod\ell$,} \\
         -1 & \mbox{if $x\in A_i$, $y\in A_j$, $\lceil i/\ell\rceil\not=\lceil j/\ell\rceil$ and $i=(j\pm 1)\bmod\ell$,} \\
	 0  & \mbox{otherwise.}
         \end{cases}\]
Let $G$ be a graph that has a vertex $v$ of degree one and let $v'$ be the neighbor of $v$.
Note that
\begin{align*}
t(G,U) & = \int_{[0,1]^{V(G)}}\prod_{uu'\in E(G)}U(x_u,x_{u'})\dd x_{V(G)} \\
       & = \int_{[0,1]^{V(G)\setminus\{v\}}}\prod_{\substack{uu'\in E(G)\\uu'\not=vv'}}U(x_u,x_{u'})\cdot \left(\int_{[0,1]}U(x_{v'},x_v)\dd x_v\right)\dd x_{V(G)\setminus\{v\}} \\
       & = \int_{[0,1]^{V(G)\setminus\{v\}}}\prod_{\substack{uu'\in E(G)\\uu'\not=vv'}}U(x_u,x_{u'})\cdot 0\dd x_{V(G)\setminus\{v\}} = 0
\end{align*}
We conclude that $t(G,U)=0$ for every graph $G$ with a vertex of degree one.

We next compute $t(C_{\ell},U)$. Observe that the product
$\prod_{i\in [\ell]}U(x_i,x_{(i+1)\bmod\ell})$ 
is non-zero for $x_1,\ldots,x_{\ell}\in [0,1]$ if and only if
there exists $j\in [\ell]$ such that
either $x_i\in A_{(i+j)\bmod\ell}\cup A_{(i+j)\bmod\ell+\ell}$ for every $i\in [\ell]$ or
$x_i\in A_{(\ell-i+j)\bmod\ell}\cup A_{(\ell-i+j)\bmod\ell+\ell}$ for every $i\in [\ell]$;
if the product is non-zero, then it is equal to one.
Hence, it follows that
\[
t(C_\ell,U) = \int_{[0,1]^\ell}\prod_{i\in [\ell]}U(x_i,x_{(i+1)\bmod\ell})\dd x_{[\ell]}
            = 2\ell\cdot\prod_{i\in [\ell]}|A_i\cup A_{i+\ell}|=\frac{2}{\ell^{\ell-1}} \]

We next consider the following graphons: $W_1=W_2=1/k+\varepsilon U$, $W_3=1/k-2\varepsilon U$ and $W_4=\cdots=W_k=1/k$.
We will estimate the homomorphism density of $H$ in $W_1,\ldots,W_k$ using Proposition~\ref{prop:epsU}.
Note that if $F$ is a subset of edges of $H$ such that $1\le |F|\le\ell$,
then $H[F]$ contains a vertex of degree one unless $H[F]$ is a union of a cycle of length $\ell$ and isolated vertices.
In particular, $t(H[F],U)=0$ for a set $F$ of $\ell$ edges unless $F$ is the edge set of a cycle of length $\ell$.
Using Proposition~\ref{prop:epsU}, we obtain that
\begin{align*}
t(H,W_1)+\cdots+t(H,W_k) & = 2t(H,1/k+\varepsilon U)+t(H,1/k-2\varepsilon U)\\ &+(k-3)t(H,1/k) \\
                         & = k^{-\|H\|+1}+2\cdot\frac{2m_{\ell}}{\ell^{\ell-1}}k^{-\|H\|+\ell}\varepsilon^{\ell}-\frac{2^{\ell+1}m_{\ell}}{\ell^{\ell-1}}k^{-\|H\|+\ell}\varepsilon^{\ell}\\ &+O(\varepsilon^{\ell+1}) \\
			 & = k^{-\|H\|+1}-\frac{(2^{\ell+1}-4)m_{\ell}}{\ell^{\ell-1}}k^{-\|H\|+\ell}\varepsilon^{\ell}+O(\varepsilon^{\ell+1})
\end{align*}
where $m_{\ell}$ is the number of cycles of length $\ell$ in $H$.
Since $2^{\ell+1}-4>0$, there exists $\varepsilon_0>0$ such that
\[t(H,W_1)+\cdots+t(H,W_k)<k^{-\|H\|+1}\]
for every $\varepsilon\in(0,\varepsilon_0)$.
We conclude that $H$ is not locally $k$-common, which completes the proof of the theorem.
\end{proof}

\section{Open problems}

We conclude with two open problems.
Theorem~\ref{thm:append} provides an example of a non-bipartite $k$-common graph for every $k\geq2$.
A natural next question is whether there exist $k$-common graphs of arbitrary large chromatic number. 
Currently, the only known example of a $2$-common graph of chromatic number greater than three is the 
$5$-wheel~\cite{HatHKNR12} and so this question is interesting even in the case $k=2$ and
$\ell\geq5$.

\begin{problem}
\label{prob:highChi}
For every $k\ge 2$ and $\ell\ge 4$, construct a $k$-common $\ell$-chromatic graph.
\end{problem}

The second problem stems from Theorem~\ref{thm:local}
which characterizes locally $k$-common graphs for $k\ge 3$.
Interestingly, we do not have a similar characterization of locally $2$-common graphs and
we even miss a natural conjecture for such a characterization.

\begin{problem}
Characterize graphs that are locally $2$-common.
\end{problem}

Locally $2$-common graphs include forests, all graphs with even girth, the triangle and the $5$-wheel in particular,
since these graphs are locally Sidorenko or $2$-common.
On the other hand, 
Cs\'oka, Hubai and Lov\'asz~\cite{CsoHL19} showed that for every graph $H$ containing $K_4$ and every $\varepsilon>0$,
there exists a kernel $U$ such that $\|U\|_{\infty}\le\varepsilon$ and
\[t(H,1/2+U)+t(H,1/2-U)<2^{-\|H\|}.\]
In particular, no graph containing $K_4$ is locally $2$-common.
We remark that the notion of locally common graphs used in~\cite{CsoHL19} is formally weaker than the notion used in this paper,
i.e., every graph locally $2$-common in the sense used in this paper is locally common in the sense used in~\cite{CsoHL19},
however, it is not obvious whether the converse holds.
For completeness, we present a simple argument that $K_4$ is not locally $2$-common,
which is based on a construction of Franek and R\"odl~\cite{FraR93} of a kernel $U$ such that
\[t(K_4,1/2+U)+t(K_4,1/2-U)\le 0.987314\times\frac{1}{32}\quad\mbox{and}\quad\int_{[0,1]}U(x,y)\dd y=0\]
for every $x\in [0,1]$.
For $z\in (0,1]$, define a kernel $U_z$ as
\[U_z(x,y)=\begin{cases}
           U(x/z,y/z) & \mbox{if $(x,y)\in [0,z]^2$,} \\
	   0 & \mbox{otherwise.}
	   \end{cases}\]
Since the cut norm of $U_z$ is at most $z^2$ and $t(K_4,1/2+U_z)+t(K_4,1/2-U_z)<1/32$ (here, we use that the kernel $U_z$
is ``$0$-regular''),
it follows that $K_4$ is not locally $2$-common.

\section*{Acknowledgement}

The authors would like to thank L\'aszl\'o Mikl\'os Lov\'asz for insightful comments on the graph limit notions discussed in this paper, and the two anonymous reviewers for preparing very detailed reports in an unusually fast way;
the comments of the reviewers helped to substantially improve the presentation of our results.

\bibliographystyle{bibstyle}
\bibliography{kcommon}

\end{document}